\documentclass[12pt, oneside]{amsart}

\usepackage{amsmath,amsfonts}
 \usepackage[english]{babel}

\newtheorem{theorem}{Theorem}[section]
\newtheorem{lemma}[theorem]{Lemma}
\newtheorem{definition}[theorem]{Definition}
\newtheorem{question}[theorem]{Question}
\newtheorem{proposition}[theorem]{Proposition}
\newtheorem{corollary}[theorem]{Corollary}
\numberwithin{equation}{section}

\begin{document}

\begin{center}
Olena Karlova, Volodymyr Mykhaylyuk

{\it Yurii Fedkovych Chernivtsi National University}

CROSS-TOPOLOGY AND LEBESGUE TRIPLES

\end{center}

\medskip
\hrule

\medskip

\section{Introduction}

Let $X$, $Y$ and $Z$ be topological spaces. For a mapping
$f:X\times Y\to Z$ and a point $(x,y)\in X\times Y$ we write $f^x(y)=f_y(x)=f(x,y)$.
A mapping $f:X\times Y\to Z$ is said to be  {\it separately continuous}, if mappings $f^x:Y\to Z$ and $f_y:X\to Z$ are continuous for all $x\in X$ and $y\in Y$.
If  $f:X\to Y$ is a pointwise limit of a sequence of separately continuous mappings $f_n:X\to Y$, then $f$ is {\it a Baire-one mapping} or $f$ {\it belongs to the first Baire class}.

In 1898 H.~Lebesgue~\cite{Leb1} proved that if $X=Y=\mathbb R$ then every separately continuous function $f:X\times Y\to Z$ belongs to the first Baire class. A collection of topological spaces $(X,Y,Z)$ with this property we call {\it a Lebesgue triple}.

The result of Lebesgue was generalized by many mathematicians (see \cite{Hahn, Ru, KMas3, BanakhT,K1} and the references given there). In particular, A.~Kalancha and V.~Maslyuchenko~\cite{KMas3} showed that $({\mathbb R},{\mathbb R},Z)$ is a Lebesgue triple if $Z$ is a topological vector space. T.~Banakh~\cite{BanakhT} proved that $({\mathbb R},{\mathbb R},Z)$ is a Lebesgue triple in the case when $Z$ is an equiconnected space.  It follows from \cite[Theorem~3]{K1} that for every metrizable arcwise connected and locally arcwise connected space $Z$ a collection $({\mathbb R},{\mathbb R},Z)$ is a Lebesgue triple.

In the connection with the above-mentioned results V.~Maslyuchenko put the following question.
\begin{question}\label{q2}
Does there exist a topological space $Z$ such that  $({\mathbb R},{\mathbb R},Z)$ is not a Lebesgue triple?
\end{question}

Here we give the positive answer to this question. Moreover, we prove that $(X,Y,Z)$ is not a Lebesgue triple for topological spaces $X$ and $Y$ of a wide class of spaces, which includes, in particular, all spaces $\mathbb R^n$, and for a space $Z=X\times Y$ endowed with the cross-topology (see definitions in Section~\ref{sec:2}). In Sections~\ref{sec:2} and \ref{sec:3} we give some auxiliary properties of this topology. Section~\ref{sec:4} contains a proof of the main result. In the last section we show that connectedness-like conditions on spaces $X$ and $Y$ in the main result are essential. We prove that $(X,Y,Z)$ is a Lebesgue triple when $X$ is a strongly zero-dimensional metrizable space, $Y$ and $Z$ are arbitrary topological spaces.

\section{Compact sets in the cross-topology}\label{sec:2}

Let  $X$ and $Y$ be topological spaces. We denote by $\gamma$  the collection of all subsets $A$ of $X\times Y$ such that for every point $(x,y)$ of $A$ there exist such neighborhoods $U$ and $V$ of $x$ and $y$ in $X$ and $Y$, respectively, that  $(\{x\}\times V)\bigcup (U\times \{y\})\subseteq A$. The system $\gamma$ forms a topology on $X\times Y$, which is called {\it the cross-topology}. A product $X\times Y$ with the cross-topology we denote by $(X\times Y, \gamma)$.

For a point $p=(x,y)\in X\times Y$ by ${\rm cross}(p)$ we denote the set $(\{x\}\times Y)\cup (X\times\{y\})$.
For an arbitrary $A\subseteq X\times Y$ let  ${\rm cross}(A)=\bigcup\limits_{p\in A}{\rm cross}(p)$.

\begin{proposition}\label{pr:1}
  Let $X$ and$Y$ be $T_1$-spaces and $(p_n)_{n=1}^\infty$ a sequence of points $p_n=(x_n,y_n)\in X\times Y$ such that $x_n\ne x_m$ and $y_n\ne y_m$ if $n\ne m$. Then $P=\{p_n:n\in\mathbb N\}$ is a $\gamma$-discrete set.
\end{proposition}

\begin{proof}
  Since every one-point subset of $X$ or of $Y$ is closed, $P$ is $\gamma$-closed. Similarly, every subset $Q\subseteq P$ is also $\gamma$-closed. Hence, $P$ is closed discrete subspace of $(X\times Y, \gamma)$.
\end{proof}

\begin{proposition}\label{pr:2}
  Let $X$ and $Y$ be $T_1$-spaces and let $K\subseteq X\times Y$ be a $\gamma$-compact set. Then there exists a countable set  $A\subseteq X\times Y$ such that $K\subseteq {\rm cross}(A)$.
\end{proposition}

\begin{proof}
 Assume that $K\not\subseteq {\rm cross}(A)$ for any finite set $A\subseteq X\times Y$. We choose an arbitrary point $p_1\in K$ and by the induction on  $n\in\mathbb N$ we construct a sequence of points $p_n\in K$ such that $p_{n+1}\in K\setminus{\rm cross}(P_n)$, where $P_n=\{p_k:1\leq k\leq n\}$ for every $n\in\mathbb N$. According to Proposition~\ref{pr:1} the set $P=\{p_n:n\in\mathbb N\}$ is infinite $\gamma$-discrete in  $K$ which contradicts the fact that $K$ is $\gamma$-compact.
\end{proof}

\begin{proposition}\label{pr:3}
  Let $X$ and $Y$ be $T_1$-spaces and let $A$ and $B$ be discrete sets in $X$ and $Y$, respectively. Then the topology of product $\tau$ and the cross-topology $\gamma$ coincide on the set $C={\rm cross}(A\times B)$.
\end{proposition}

\begin{proof}
Fix $p=(x,y)\in C$. Using the discreteness of $A$ and $B$, we choose such neighborhoods $U$ and $V$ of  $x$ and $y$ in $X$ and $Y$, respectively, that $|U\cap A|\leq 1$ and $|V\cap B|\leq 1$. Then $C\cap (U\times V)= C\cap {\rm cross}(c)$ for some point $c\in C$. Then $\tau=\gamma$ on the set $C\cap (U\times V)$.
\end{proof}

Propositions~\ref{pr:2} and \ref{pr:3} immediately imply the following characterization of $\gamma$-compact sets.
\begin{proposition}\label{pr:4}
  Let $X$ and $Y$ be $T_1$-spaces and $K\subseteq X\times Y$. Then $K$ is $\gamma$-compact if and only if when
  \begin{enumerate}
    \item   $K$ is compact;

    \item  $K\subseteq {\rm cross}(C)$ for a finite set $C\subseteq X\times Y$.
  \end{enumerate}
\end{proposition}

\section{Connected sets and cross-mappings}\label{sec:3}

\begin{proposition}\label{pr:3.1}
Let $X$ and $Y$ be connected spaces, $A$ a dense subset of $X$, let $\O\ne B\subseteq Y$ and $C\subseteq X\times Y$ be such sets that ${\rm cross}(A\times B)\subseteq C$. Then $C$ is connected.
\end{proposition}

\begin{proof}
Let $U$ and $V$ be open subsets of $C$ such that  $C=U\sqcup V$. Since $X$ and $Y$ are connected, for every $p\in A\times B$ either ${\rm cross}(p)\subseteq U$, or ${\rm cross}(p)\subseteq V$. Since ${\rm cross}(p)\cap {\rm cross}(q)\ne\emptyset$ for distinct points  $p,q\in X\times Y$, ${\rm cross}(A \times B)\subseteq U$ or ${\rm cross}(A\times B)\subseteq V$. Taking into account that ${\rm cross}(A\times B)$ is dense in $X\times Y$, and consequently, in $C$, we obtain that  $C\subseteq U$ or $C\subseteq V$. Therefore, $U=\emptyset$ or $V=\emptyset$. Hence, $C$ is connected.
\end{proof}

\begin{corollary}\label{pr:3.2}
Let $X$ and $Y$ be infinite connected $T_1$-spaces. Then the complement to any finite subset of  $X\times Y$ is connected.
\end{corollary}

\begin{proof}
Let $C\subseteq X\times Y$ be a finite set. We choose finite sets  $A\subseteq X$ and $B\subseteq Y$ such that $C\subseteq A\times B$. Remark that $A_1=X\setminus A$ and $B_1=Y\setminus B$ are dense in $X$ and $Y$, respectively, and ${\rm cross}(A_1\times B_1)\subseteq (X\times Y)\setminus C$. It remains to apply Proposition~\ref{pr:3.1}.
\end{proof}

\begin{definition}
  {\rm A topological space $X$ is said to be {\it a $C_1$-space} (or {\it a space with the property $C_1$}),  if the complement to any finite subset has finite many components.}
\end{definition}

Let us observe that the real line ${\mathbb R}$ has the property $C_1$. Moreover, a finite product of $C_1$-spaces is a $C_1$-space.

Let $X$ and $Y$ be topological spaces and $P\subseteq X\times Y$. A mapping $f:P\to X\times Y$ is called {\it a cross-mapping}, if $f(p)\subseteq {\rm cross}(p)$ for every $p\in P$.

\begin{lemma}\label{l:2}
  Let $X$ and $Y$ be Hausdorff spaces, $U\subseteq X$, $V\subseteq Y$, let $f:U\times V\to X\times Y$ be a continuous cross-mapping, let $A\subseteq X$ and $B\subseteq Y$ be finite sets and the following conditions hold:
   \begin{enumerate}
   \item $U$, $V$ be connected $C_1$-spaces;

   \item  $f(U\times V)\subseteq {\rm cross}(A\times B)$.
   \end{enumerate}

  Then either $f(U\times V)\subseteq \{a\}\times Y$ for some $a\in A$, or  $f(U\times V)\subseteq X\times \{b\}$ for some $b\in B$.
\end{lemma}

\begin{proof} If both $U$ and $V$ are finite, then (1) imply that $U$ and $V$ are one-point sets and the lemma follows from~(2). If $U$ is finite (one-point) and $V$ is infinite, then $F=\{z\in U\times V: f(z)\in {\rm cross}(A\times B)\setminus (A\times Y)\}$ is finite clopen subset of  $U\times V$. The connectedness of  $U\times V$ implies $F=\O$. Hence, $f(U\times V)\subseteq A\times Y$. Since $f$ is continuous,  $f(U\times V)\subseteq \{a\}\times Y$ for some $a\in A$.

Now let $U$ and $V$ be infinite. Then it follows from $(1)$ that $U$ and $V$ have no isolated points. Since $A_1=A\cap U$ and $B_1=B\cap V$ are closed and nowhere dense in $U$ and $V$, respectively, the set $$C=(U\times V)\cap{\rm cross}(A\times B)=(U\times V)\cap{\rm cross}(A_1\times B_1)$$ is closed and nowhere dense in $Z=U\times V$.

Let $\alpha:U\times V\to X$ and $\beta:U\times V\to Y$ be continuous functions such that $f(x,y)=(\alpha(x,y),\beta(x,y))$ for all $(x,y)\in Z$. Put
  $$
  Z_\alpha=\{(x,y)\in Z:\alpha(x,y)=x\}, \,\,\, Z_\beta=\{(x,y)\in Z:\beta(x,y)=y\}.
  $$

 Notice that $$P_\alpha=\{z\in Z_\alpha:\alpha(z)\in A\}=Z_\alpha\cap(A\times Y)=Z_\alpha\cap(A_1\times Y)$$ is nowhere dense in $Z$. Therefore, $Q_\alpha=\{z\in Z_\alpha:\alpha(z)\not\in A\}$ is dense in ${\rm int}_Z(Z_\alpha)$, where by ${\rm int}_Z(D)$ we denote the interior of $D\subseteq Z$ in $Z$, and by $\overline{D}$ we denote the closure of $D$ in $Z$.  Condition~$(2)$ implies that $Q_\alpha$ is contained in the closed set $\{z\in Z: \beta(z)\in B\}$. Hence, $$\overline{{\rm int}_Z(Z_\alpha)}\subseteq \overline{Q_\alpha}\subseteq \{z\in Z:\beta(z)\in B\},$$ i.e. $f(\overline{{\rm int}_Z(Z_\alpha)})\subseteq X\times B$. Similarly, $f(\overline{{\rm int}_Z(Z_\beta)})\subseteq A\times Y$.

  Since $f$ is a cross-mapping, $Z=Z_\alpha\cup Z_\beta$. Remark that $Z_\alpha$ and $Z_\beta$ are closed in $Z$. Let $$G=Z\setminus C.$$ Taking into account that  $C$ is closed and nowhere dense in $Z$, we have that $G$ is open and dense in $Z$. According to~(1) the sets $U\setminus A$ and $V\setminus B$ has finite many components, therefore, the set $G=(U\setminus A)\times (V\setminus B)$ has finite many components $G_1,\dots,G_k$. Then $G=\bigsqcup\limits_{i=1}^{k}G_i$, where all $G_i$ are closed in $G$. Hence, all the sets $G_i$ are clopen in $G$, in particular, open in $Z$. Notice that $Z_\alpha\cap Z_\beta=\{z\in Z: f(z)=z\}\subseteq f(Z)\subseteq {\rm cross}(A\times B)$. Thus, $G\cap Z_\alpha\cap Z_\beta=\emptyset$. Then $G_i\subseteq (Z_\alpha\cap G_i)\sqcup(Z_\beta\cap G_i)$, consequently $G_i\subseteq Z_\alpha$  or $G_i\subseteq Z_\beta$ for every $1\le i\le k$. Let
  $$
  I_\alpha=\{1\le i\le k: G_i\subseteq Z_\alpha\},\,\,\,   I_\beta=\{1\le i\le k: G_i\subseteq Z_\beta\},
  $$
  $$
  U_\alpha=\bigcup\limits_{i\in I_\alpha}G_i,\,\,\, U_\beta=\bigcup\limits_{i\in I_\beta} G_i.
  $$
  Remark that
  $$f(\overline{U_\alpha})\subseteq f(\overline{{\rm int}_Z(Z_\alpha)})\subseteq X\times B,\,\,\, f(\overline{U_\beta})\subseteq f(\overline{{\rm int}_Z(Z_\beta)})\subseteq A\times Y.$$
  Hence, for any $z=(x,y)\in \overline{U_\alpha}$ we have $\alpha(x,y)=x$ and $\beta(x,y)\in B$. Similarly, $\alpha(x,y)\in A$ and $\beta(x,y)=y$ for every $z=(x,y)\in \overline{U_\beta}$. Therefore, $z=f(z)\in A\times B$ for any $z\in \overline{U_\alpha}\cap\overline{U_\beta}$. Consequently, $Z_0=\overline{U_\alpha}\cap\overline{U_\beta}$ is finite.

  Denote $E=\overline{U_\alpha}\setminus Z_0$ and $D=\overline{U_\beta}\setminus Z_0$. Since by Proposition~\ref{pr:3.1} the set $Z\setminus Z_0$ is connected, nonempty and $Z\setminus Z_0=E\sqcup D$, taking into account that $\overline{E}\cap D=\O$ and $E\cap\overline{D}=\O$, we obtain $E=\O$ or $D=\O$. Assume that $E=\O$. Then $U_\beta$ is dense in $Z$ and $$f(Z)\subseteq f(\overline{U_\beta})\subseteq A\times Y.$$
  Since $U\times V$ is connected, $f(U\times V)$ is connected too, therefore, there is such $a\in A$ that $f(U\times V)\subseteq \{a\}\times Y$.
\end{proof}

\section{The main result}\label{sec:4}

\begin{proposition}\label{pr:4.1}
  Let  $X$ and $Y$ be $T_1$-spaces, $z_0\in X\times Y$ and let $(z_n)_{n=1}^{\infty}$ be $\gamma$-convergent to $z_0$ sequence of points $z_n=(x_n,y_n)\in X\times Y$. Then there exists $m\in\mathbb N$ such that $z_n\in {\rm cross}(z_0)$ for all $n\geq m$.
\end{proposition}

\begin{proof}
 Assume the contrary. Then by the induction on $k\in\mathbb N$ it is easy to construct a strictly increasing sequence of numbers $n_k\in\mathbb N$ such that $x_{n_i}\ne x_{n_j}$ and $y_{n_i}\ne y_{n_j}$ for distinct $i,j\in\mathbb N$, and $z_{n_k}\not\in {\rm cross}(z_0)$ for all $k\in\mathbb N$. Now the sequence $(p_k)_{k=1}^{\infty}$ of points $p_k=z_{n_k}$ converges to $z_0$, and from the other side the set $G=(X\times Y) \setminus\{p_k:k\in\mathbb N\}$ is a neighborhood of  $z_0$, a contradiction.
\end{proof}

A system $\mathcal{A}$ of subsets of a topological space $X$ is called a {\it $\pi$-pseudobase}~\cite{Tall},  if for every nonempty open set $U\subseteq X$  there exists a set $A\in \mathcal A$ such that ${\rm int}(A)\ne \emptyset$ and $A\subseteq U$.

\begin{theorem}
Let $X$ and $Y$ be Hausdorff spaces without isolated points and let $X$ and $Y$ have $\pi$-pseudobases which consist of connected compact $C_1$-sets,
and let \mbox{$f:X\times Y\to X\times Y$} be the identical mapping. Then $f\not\in B_1(X\times Y,(X\times Y,\gamma))$.
\end{theorem}

\begin{proof}
Assuming the contrary, we choose a sequence of continuous functions $f_n:X\times Y\to (X\times Y,\gamma)$ such that $f_n(x,y)\to (x,y)$ in $(X\times Y,\gamma)$ for all $(x,y)\in X\times Y$.

Remark that every $f_n:X\times Y\to X\times Y$ is continuous. Then for every  $n\in\mathbb N$ the set $P_n=\{p\in X\times Y: f_n(p)\in {\rm cross}(p)\}$ is closed. Hence, for every $n\in\mathbb N$ the set
$$
  F_n=\bigcap\limits_{m\geq n}P_m=\{p\in X\times Y: \forall m\ge n\,\, f_m(p)\in {\rm cross}(p) \}.
  $$
is closed too.  Moreover, by Proposition~\ref{pr:4.1}
  $$
  X\times Y=\bigcup\limits_{n=1}^\infty F_n.
  $$

The conditions of the theorem imply that $Z=X\times Y$ has a $\pi$-pseudobase of compact sets. Then the product contains an open everywhere dense locally compact subspace, in particular, the product $X\times Y$ is Baire. We choose a number $n_0\in\mathbb N$ and compact connected  $C_1$-sets $U\subseteq X$ and $V\subseteq Y$ such that  $U\times V\subseteq F_{n_0}$, $U_0={\rm int}(U)\ne \emptyset$ and $V_0={\rm int}(V)\ne \emptyset$.

Let $W=U\times V$. According to Proposition~\ref{pr:4}, there exist such sequences of finite sets $A_n\subseteq X$ and $B_n\subseteq Y$ that $f_n(W)\subseteq (A_n\times Y)\cup (X\times B_n)$ for every $n\in\mathbb N$. Since $X$ and $Y$ have no isolated points, the sets $U_0$ and $V_0$ are infinite. Take such points $p_1=(x_1,y_1), p_2=(x_2,y_2)\in U_0\times V_0$ that $p_1\not\in {\rm cross}(p_2)$. Since $X$ and $Y$ are Hausdorff, there exist neighborhoods $U_1$ and $U_2$ of $x_1$ and $x_2$ in $U_0$ and neighborhoods $V_1$ and $V_2$ of $y_1$ and $y_2$ in $V_0$, respectively, such that  $U_1\cap U_2= V_1\cap V_2=\O$. Now we choose a number $N\ge n_0$ such that  $f_N(p_1)\in U_1\times V_1$ and $f_N(p_2)\in U_2\times V_2$.

    Since $f_N|_W$ is a cross-mapping, Lemma~\ref{l:2} implies that $f_N(W)\subseteq \{a\}\times Y$ for some $a\in A$ or $f_N(W)\subseteq X\times \{b\}$ for some $b\in B$. Assume that $f_N(W)\subseteq \{a\}\times Y$  for some $a\in X$. Then $(U_1\times V_1)\cap (\{a\}\times Y)\ne\O$ and $(U_2\times V_2)\cap (\{a\}\times Y)\ne\O$. Therefore, $a\in U_1\cap U_2$, which is impossible.
  \end{proof}

\begin{corollary}
  Let  $n,m\ge 1$ and let $f:{\mathbb R}^n\times {\mathbb R}^m\to {\mathbb R}^n\times {\mathbb R}^m$ be the identical mapping. Then $f\not\in B_1({\mathbb R}^n\times {\mathbb R}^m,({\mathbb R}^n\times {\mathbb R}^m,\gamma))$.
\end{corollary}

\begin{corollary}
   The collection $({\mathbb R}^n,{\mathbb R}^m,({\mathbb R}^n\times {\mathbb R}^m,\gamma))$ is not a Lebesgue triple for all  $n,m\ge 1$.
\end{corollary}

\section{Separately continuous mappings on zero-dimensional spaces}\label{sec:5}

Recall that a nonempty topological space $X$ is {\it strongly zero-dimensional}, if it is completely regular and every finite functionally open cover of $X$ has a finite disjoint open refinement~\cite[с.~529]{Eng}.

\begin{theorem}
    Let $X$ be a strongly zero-dimensional metrizable space, let $Y$ and $Z$ be topological spaces. Then $(X,Y,Z)$ is a Lebesgue triple.
\end{theorem}

\begin{proof} Let $d$ be a metric on $X$, which generates its topology.
For every $n\in\mathbb N$ we consider an open cover $\mathcal B_n$ of $X$ by balls of the diameter $\le \frac 1n$. It follow from \cite{Ellis} that every $\mathcal B_n$ has locally finite clopen refinement $\mathcal U_n=(U_{\alpha,n}:0\le\alpha<\beta_n)$. For all $n\in\mathbb N$ let $V_{0,n}=U_{0,n}$ and $V_{\alpha,n}=U_{\alpha,n}\setminus \bigcup\limits_{\xi<\alpha}U_{\xi,n}$ if $\alpha>0$. Then $\mathcal V_n=(V_{\alpha,n}:0\le \alpha<\beta_n)$ is a locally finite disjoint cover of $X$ by clopen sets  $V_{\alpha,n}$ which refines  $\mathcal B_n$.

 Let $f:X\times Y\to Z$ be a separately continuous function. For all $n\in\mathbb N$ and $0\le\alpha<\beta_n$ we choose a point $x_{\alpha,n}\in V_{\alpha,n}$. Let us consider functions $f_n:X\times Y\to Z$ defined as the following:
  $$
  f_n(x,y)=f(x_{\alpha,n},y),
  $$
  if $x\in V_{\alpha,n}$ and $y\in Y$. Clearly, for every $n\in\mathbb N$ the function $f_n$ is jointly continuous, provided  $f$ is continuous with respect to the second variable. We show that $f_n(x,y)\to f(x,y)$ on $X\times Y$. Fix $(x,y)\in X\times Y$ and choose a sequence $(\alpha_n)_{n=1}^\infty$ such that  $x\in V_{\alpha_n,n}$. Since ${\rm diam}V_{\alpha_n,n}\to 0$,  $x_{\alpha_n,n}\to x$. Taking into account that $f$ is continuous with respect to the first variable, we obtain that
  $$
  f_n(x,y)=f(x_{\alpha_n,n},y)\to f(x,y).
  $$
  Hence, $f\in B_1(X\times Y,Z)$.
  \end{proof}

\end{document}